\documentclass[12pt]{amsart}
\usepackage{amssymb, colordvi}
\usepackage{multirow}
\usepackage{graphicx,rotating}
\usepackage{algorithmic}
\usepackage{color}
\usepackage{appendix}

\numberwithin{equation}{section}
\newtheorem{theorem}{Theorem}
\numberwithin{theorem}{section}
\newtheorem{proposition}[theorem]{Proposition}
\newtheorem{lemma}[theorem]{Lemma}
\newtheorem{corollary}[theorem]{Corollary}
\newtheorem{assumption}[theorem]{Assumption}
\newtheorem{ex}[theorem]{Example}
\newtheorem{rem}[theorem]{Remark}

\newtheorem{alg}[theorem]{Algorithm}

\newenvironment{example}{\begin{ex}\rm}{\end{ex}}
\newenvironment{remark}{\begin{rem}\rm}{\end{rem}}
\newenvironment{algorithm}{\begin{alg}\rm}{\end{alg}}

\newcounter{FNC}[page]
\def\fauxfootnote#1{{\addtocounter{FNC}{2}$^\fnsymbol{FNC}$%
     \let\thefootnote\relax\footnotetext{\Maroon{$^\fnsymbol{FNC}$#1}}}}

\def\rgbColor#1#2{#2}

\newcommand{\defcolor}[1]{\Blue{#1}}
\newcommand{\demph}[1]{{\it \defcolor{#1}}}

\newcommand\codim{\operatorname{codim}}

\newcommand\ubc{\mbox{\rm ubc}}

\newcommand{\calA}{{\mathcal A}}
\newcommand{\calB}{{\mathcal B}}

\newcommand{\calV}{{\mathcal V}}

\newcommand{\oD}{\overline{\Delta}}
\newcommand{\op}{\overline{p}}

\newcommand{\oy}{\overline{y}}

\newcommand{\C}{\mathbb{C}}
\newcommand{\R}{\mathbb{R}}
\newcommand{\Z}{{\mathbb Z}}

\title[Software for the Gale transform of fewnomial systems]%
  {Software for the Gale transform of fewnomial systems and a Descartes rule for fewnomials}

\author[D.~J.~Bates]{Daniel J.~Bates}
\address{Department of Mathematics\\
         Colorado State University\\
         Fort Collins, CO 80523-1874 \\
         USA}
\email{bates@math.colostate.edu}
\urladdr{http://www.math.colostate.edu/~bates/}

\author[J.~D.~Hauenstein]{Jonathan D.~Hauenstein}
\address{Department of Applied \& Computational Mathematics \& Statistics\\
         University of Notre Dame\\
         Notre Dame, IN  46556\\         
         USA}
\email{hauenstein@nd.edu}
\urladdr{http://www.nd.edu/~jhauenst}

\author[M.~E.~Niemerg]{Matthew E. Niemerg}
\address{Institute for Interdisciplinary Information Sciences\\
04-603 FIT Building\\
Tsinghua University\\
Beijing, China 100084
}
\email{research@matthewniemerg.com}
\urladdr{http://www.matthewniemerg.com}

\author[F.~Sottile]{Frank Sottile}
\address{Department of Mathematics\\
         Texas A\&M University\\
         College Station, TX 77843\\
         USA}
\email{sottile@math.tamu.edu}
\urladdr{http://www.math.tamu.edu/\~{}sottile/}

\thanks{Bates and Niemerg supported by NSF grants DMS-0914674 and DMS-1115668.}
\thanks{Niemerg supported in part by the
National Basic Research Program of China Grants 2011CBA00300, 2011CBA00301, 
the Natural Science Foundation of China Grants 61044002, 61361136003.}
\thanks{Hauenstein supported by NSF grant DMS-1262428, DARPA YFA, and Sloan Research Fellowship.}
\thanks{Sottile supported by the NSF grants DMS-0915211 and DMS-1001615.}  
\thanks{All authors supported by Institut Mittag-Leffler.}



\begin{document}

\begin{abstract}
 We give a Descartes'-like bound on the number of positive solutions to a system of fewnomials that holds when
 its exponent vectors are not in convex position and a sign condition is satisfied.
 This was discovered while developing algorithms and software for computing the Gale transform of a fewnomial
 system, which is our main goal.
 This software is a component of a package we are developing for Khovanskii-Rolle continuation,
 which is a numerical algorithm to compute the real solutions to a system of fewnomials.
\end{abstract}

\maketitle

Determining the {\it real} solutions to a system of polynomial equations is a challenging
problem with real-world applications.
Typically, few of the complex solutions are real and 
even fewer satisfy meaningful constraints such as positivity.
Methods based on symbolic computation~\cite{RUR} or numerical
homotopy continuation~\cite{BertiniBook,SW05} first find all complex solutions.  It is an active
area of research to develop methods that find only the meaningful real solutions.
Among these are exclusion methods~\cite{Ge01,Ge03}, methods based on semidefinite
programming~\cite{LLR}, the critical point method~\cite{H,RRS}, and one proposed by two of
us~\cite{BS10} that is based on the Khovanskii-Rolle Theorem~\cite{Kh91} from the theory of
fewnomials. 

A system of polynomial equations with few monomials is a \demph{fewnomial system}.
Such a system enjoys bounds on its number of real solutions independent of its number of complex 
solutions~\cite{BBS,BS,Kh91}. 
Khovanskii-Rolle continuation is a symbolic-numeric method that can find all
positive real solutions of a fewnomial system without first finding all complex solutions.
In fact, it is the first algorithm for computing positive solutions with complexity 
depending primarily on the fewnomial bound on the number of real solutions, rather than the number
of complex solutions or the ambient dimension.
The core algorithm was described in~\cite{BS10}, along with a 
proof of concept implementation.

Khovanskii-Rolle continuation rests upon the fundamental notion of 
\demph{Gale duality for polynomial systems}, which is a scheme-theoretic isomorphism between 
complete intersections of Laurent polynomials in an algebraic torus and complete intersections
of master functions in the complement of a hyperplane
arrangement~\cite{BS_Gale}. 
The support of the polynomials annihilates the weights of the master functions,
whence the term Gale duality.

The Khovanskii-Rolle continuation algorithm finds all real solutions to a system of master
functions.  Computing approximations to the real solutions of a fewnomial system $F$ requires that $F$ 
be converted into its dual Gale system $G$ which is solved using Khovanskii-Rolle continuation, 
then the numerical approximations to solutions of $G$ are converted back to
numerical approximations to solutions of the original fewnomial system $F$.
Symbolic choices made in constructing the dual system affect the performance of the Khovanskii-Rolle continuation
algorithm.   

One such choice is that of a basis for the nullspace of the matrix of exponent vectors of
the monomials of $F$.
When the nullspace meets the positive orthant, a bound lower than the fewnomial bound holds.
Choosing a positive basis element results in far fewer paths to be followed in the Khovanskii-Rolle continuation 
algorithm.   
This happens when the exponent vectors of the original polynomial system are not in convex
position and a sign condition holds on the coefficient matrix.
Such sign conditions are also found in recent work giving very strong bounds on positive
solutions~\cite{BD,Craciun,MFRCSD} and are considered to be multivariate versions of Descartes' rule of signs.

In addition to this new Descartes'-like bound,
we make explicit the algorithms for both the creation 
of $G$ from $F$ and the computation of solutions of $F$ from solutions of $G$.
These steps were described 
existentially in~\cite{BS_Gale}.
We have implemented these algorithms in a software package {\tt galeDuality}, which is
available at each authors' website.
This will be the front end for a software package that we are developing for Khovanskii-Rolle continuation.

Section~\ref{S:KhRo} gives a description of Gale duality and sketches Khovanskii-Rolle continuation,
illustrating it through an extended example.
In Section~\ref{S:bound}, we present our Descartes'-like bound for fewnomial systems and establish a key lemma about
the location of points to be tracked in the algorithm.
Explicit pseudocode for using Gale duality to compute master functions
(Algorithm~\ref{alg:wrapping}) and to transform approximations to solutions of the Gale system into
approximations to solutions of the fewnomial system (Algorithm~\ref{alg:unwrapping}) is provided
in Section~\ref{S:algos}, along with details for portions of these algorithms.
One subroutine of Algorithm~\ref{alg:wrapping} requiring more care is covered in Subsection~\ref{S:null}. 
Finally, a brief description of our software package is provided in Section~\ref{S:software}.

%
\section{Khovanskii-Rolle continuation and the Gale transform}\label{S:KhRo}
%

We give an example of Gale duality for polynomial systems and the Khovanskii-Rolle
continuation algorithm.
For a complete treatment, see~\cite{BS10,BS_Gale}.

%
\subsection{An example of Gale duality}\label{SS:Gale_algebra}

Consider the system of Laurent polynomials,
 \begin{eqnarray}
    v^2w^3 - 11uvw^3   - 33uv^2w + 4v^2w + 15u^2v + 7 &=& 0\,,\nonumber\\
    v^2w^3\hspace{53pt}+\hspace{5.6pt}5uv^2w 
         - 4v^2w -\hspace{5.6pt}3u^2v + 1 &=& 0\,,\label{Eq:system1}\\
    v^2w^3 - 11uvw^3   - 31uv^2w + 2v^2w + 13u^2v + 8 &=& 0\,.\nonumber
 \end{eqnarray}
This system has 20 complex solutions.
Six are real and three are positive,
\[ 
  ( 1.194,  0.374,  1.231)\,,\ 
  ( 0.431,  0.797,  0.972)\,,\ 
  ( 0.613,  0.788,  0.850)
  \,.
\]
Solving the equations in~\eqref{Eq:system1} for the monomials $v^2w^3$, $v^2w$, and $uvw^3$, 
gives
 \begin{eqnarray}
    v^2w^3 &=&               1    - u^2v - uv^2w\,,\nonumber\\   \label{Eq:system2}
    v^2w   &=&       \tfrac{1}{2} - u^2v + uv^2w\,,\qquad\mbox{and}\\
   uv  w^3 &=&\tfrac{10}{11}(1    + u^2v - 3uv^2w)\,.\nonumber
 \end{eqnarray}
Since 
 \begin{eqnarray*}
  \left(uv^2w\right)^3 \cdot \left(v^2w^3\right) &=& u^3v^8w^6\ =\ 
   \left(u^2v\right)\cdot \left(v^2w\right)^3\cdot \left(uvw^3\right)\,
\qquad\mbox{and}\\
  \left(u^2v\right)^2 \cdot \left(v^2w^3\right)^3
   &=& u^4v^8w^9\ =\ 
  \left(uv^2w\right)^2\cdot \left(v^2w\right)\cdot \left(uvw^3\right)^2\,,
 \end{eqnarray*}
we may substitute the expressions on the right hand sides of~\eqref{Eq:system2} 
for the monomials $v^2w^3$, $v^2w$, and $uvw^3$ in these expressions to obtain the system 
 \begin{eqnarray*}
  \left(uv^2w\right)^3 \cdot \left(1 - u^2v - uv^2w\right)
   &=&  \left(u^2v\right)\cdot \left(\tfrac{1}{2} - u^2v + uv^2w\right)^3
       \cdot \left(\tfrac{10}{11}(1+u^2v - 3uv^2w)\right)\\
  \left(u^2v\right)^2 \cdot \left(1 - u^2v - uv^2w\right)^3
   &=& 
  \left(uv^2w\right)^2\cdot \left(\tfrac{1}{2} - u^2v + uv^2w\right)
        \cdot \left(\tfrac{10}{11}(1+u^2v - 3uv^2w)\right)^2\,.
 \end{eqnarray*}
Writing \defcolor{$x$} for $u^2v$ and \defcolor{$y$} for $uv^2w$ and solving for $0$, these become
 \begin{equation}\label{Eq:Gale1}
  \begin{array}{rcrcl}
   \Blue{g}&:=&
       y^3(1-x-y)\ \ -\ \       x(\tfrac{1}{2}-x+y)^3\left(\tfrac{10}{11}(1+x-3y)\right) &=&
        0\,,\quad\mbox{and}\\ 
   \ForestGreen{f}&:=&
        x^2(1-x-y)^3\ \ -\ \ y^2(\tfrac{1}{2}-x+y)\left(\tfrac{10}{11}(1+x-3y)\right)^2 &=&
                  0\,.\rule{0pt}{15pt} 
  \end{array}
 \end{equation}
This system also has 20 solutions with six real, outside of the five lines where
the degree one factors vanish. 
Figure~\ref{F:Gale_pic1} shows the curves the equations~\eqref{Eq:Gale1} define and the  
five lines.
\begin{figure}[htb]
   \begin{picture}(280,180)
    \put(0,0){\includegraphics[height=180pt]{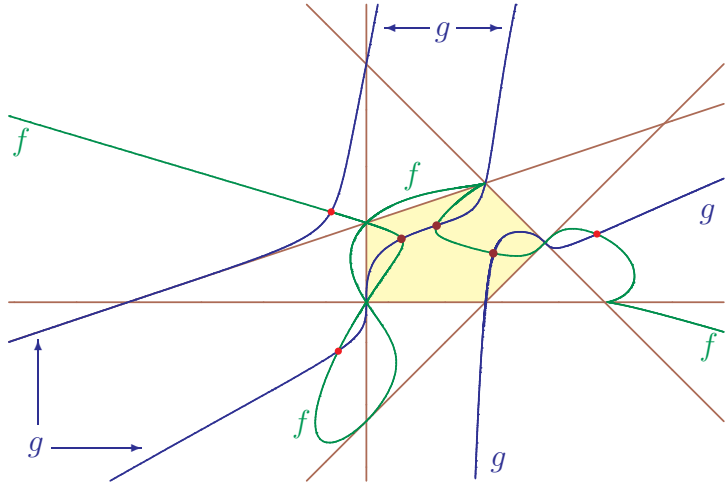}}
    \put(2,126){$\ForestGreen{f}$}    \put(108,19){$\ForestGreen{f}$}
    \put(262,48){$\ForestGreen{f}$}   \put(150,113){$\ForestGreen{f}$}
    \put( 8,12){$\Blue{g}$}  \put(183,6){$\Blue{g}$}
    \put(162,169){$\Blue{g}$}  \put(262,100){$\Blue{g}$}
    \put(17,13){\Blue{\vector(1,0){34}}}
    \put(12,22){\Blue{\vector(0,1){32}}}

    \put(160,172){\Blue{\vector(-1,0){17}}}
    \put(170,172){\Blue{\vector(1,0){17}}}

   \end{picture} 
\caption{Curves and lines}\label{F:Gale_pic1}
\end{figure}
Three solutions lie in the pentagon.
They correspond to the positive solutions of~\eqref{Eq:system1}.

This transformation from the original polynomial system~\eqref{Eq:system1} to the
system~\eqref{Eq:Gale1} is called \demph{Gale duality for polynomial systems}.
The second system has three forms which are equivalent in the pentagon.
We first rewrite~\eqref{Eq:Gale1} as a system of rational master functions
 \begin{equation}\label{Eq:Master1}
  \frac{x^2(1-x-y)^3}{y^2(\tfrac{1}{2}-x+y)\left(\tfrac{10}{11}(1+x-3y)\right)^2}
    \ =\ 
  \frac{y^3(1-x-y)}{x(\tfrac{1}{2}-x+y)^3\left(\tfrac{10}{11}(1+x-3y)\right)}
    \ =\ 1\,.
 \end{equation}
Then we take logarithms to obtain the last form.
 \begin{equation}\label{Eq:Logs1}
  \begin{array}{rcl}
   2\log(x) + 3\log(1{-}x{-}y) - 2\log(y) - \log(\tfrac{1}{2}{-}x{+}y)
   - 2 \log\left(\tfrac{10}{11}(1{+}x{-}3y)\right)& =& 0\\ \rule{0pt}{15pt}
   3\log(y) + \log(1{-}x{-}y) - \log(x) - 3\log(\tfrac{1}{2}{-}x{+}y)
   - \log\left(\tfrac{10}{11}(1{+}x{-}3y)\right)& =& 0
  \end{array}
 \end{equation}
%

%
\subsection{Geometry of Gale duality}\label{SS:Geometry}
%

Let $\R^\times = \R/\{0\}$ be the nonzero real numbers. 
Suppose that $F\colon (\R^\times)^n\to\R^n$ is given by Laurent polynomials that have a total of
$n{+}\ell$ distinct monomials, none of which is a constant.
For $b\in\R^n$, the system $F(x)=b$ of Laurent polynomials has an equivalent expression.
Let $\calA:=\{a_1,\dotsc,a_{n+\ell}\}\subset\Z^n$ be the set of $n{+}\ell$ exponents
of the monomials in $F$, which we call the \demph{support} of $F$.
Consider the map 
 \begin{equation}\label{Eq:varphi}
  \begin{array}{rcl}
   \varphi_\calA\ \colon\ (\R^\times)^n &\longrightarrow&
     (\R^\times)^{n+\ell}\ \subset\ \R^{n+\ell}\\ \rule{0pt}{14pt}
       x&\longmapsto&(x^{a_1}, x^{a_2},\dotsc,x^{a_{n+\ell}})\,.
  \end{array}
 \end{equation}
Affine-linear forms $L(z)=\lambda$ pull back to polynomials with support $\calA$, and so 
$F(x)=b$ is the pullback of a codimension $n$ (dimension $\ell$) affine linear space $H$
of $\R^{n+\ell}$ defined by $C(z)=b$, where $C\in\R^{n\times(n+\ell)}$ is the coefficient matrix of
$F$. 
We define \defcolor{$S_F$} to be the set of solutions of $F(x)=b$.  
Then,
 \begin{equation}\label{Eq:maps_between}
   \varphi_\calA(S_F)\ =\ 
   \varphi_\calA((\R^\times)^n)\cap H\,.
 \end{equation}
Let $\psi\colon\R^\ell\to H$ be any affine isomorphism.
This is given by a point $\lambda=\psi(0)\in H$ 
and $n{+}\ell$ linear forms $L_1,\dotsc,L_{n+\ell}$ on $\R^\ell$ which span the dual space of $\R^\ell$,
 \begin{equation}\label{Eq:psi}
   \psi(y)\ =\ 
    (L_1(y)+\lambda_1,\, L_2(y)+\lambda_2\,,\dotsc,\, L_{n+\ell}(y)+\lambda_{n{+}\ell})\ \in\
      H\,,\ \mbox{ for}\ y \in \R^\ell. 
 \end{equation}
The forms $L_1,\dotsc,L_{n+\ell}$ span the kernel of $C$, 
$C(L_1(y),\dotsc,L_{n+\ell}(y))=0$.  
Define 
$$\defcolor{S_G}:=\psi^{-1}(\varphi_\calA(S_F))\subset\R^\ell.$$
This set (which is isomorphic to $S_F$ as a scheme)
satisfies a simple system of equations~in~$\R^\ell$. 

Any integer linear relation with coefficients $\beta_i\in\Z$ among the exponents in $\calA$,
\[
   0\ =\ \sum_{i=1}^{n+\ell}\beta_i a_i\,,
\]
gives a monomial equation that is satisfied on $\varphi_\calA((\R^\times)^n)$ in $(\R^\times)^{n+\ell}$,
 \begin{equation}\label{Eq:toric}
   \prod_{i=1}^{n+\ell} z_i^{\beta_i}\ =\ 1\,.
 \end{equation}
This pulls back under $\psi$ to an equation involving a rational master function,
 \begin{equation}\label{Eq:Master_gale}
   \prod_{i=1}^{n+\ell} (L_i(y)+\lambda_i)^{\beta_i}\ =\ 1\,,
 \end{equation}
which is defined in the complement of the hyperplanes given by the 
affine forms $L_i(y)+\lambda_i$.

Writing elements of $\calA$ as column vectors yields an $n\times(n{+}\ell)$ integer matrix, also written~$\calA$.  
An element of the integer nullspace of $\calA$ gives a monomial
equation~\eqref{Eq:toric} vanishing on~$\varphi_\calA((\R^\times)^n)$.
Choosing a $\Z$-basis for this nullspace gives a set \defcolor{$G$} of $\ell$ equations
involving master functions~\eqref{Eq:Master_gale}, and these define the set $S_G$.
Assuming there are finitely many solutions to $F(x)=b$ in $(\C^\times)^n$, these $\ell$ monomial
equations~\eqref{Eq:toric} that cut out the torus~$\varphi_\calA((\C^\times)^n)$ in $(\C^\times)^{n+\ell}$ restrict to a
complete intersection in $H$ and any $r$ of these master functions defines a subset in the hyperplane complement of
codimension $r$.
We call this set $G$ of $\ell$ master functions a \demph{Gale system} dual to the system $F(x)=b$ of
Laurent polynomials.
We summarize some results in~\cite{BS,BS_Gale}.

\begin{proposition}\label{P:Gale}
  Let $F(x)=b$ be a system of $n$ Laurent polynomials on $(\R^\times)^n$ with support~$\calA$
  consisting of $n{+}\ell$ nonzero exponents of monomials expressed as a matrix
  \mbox{$\calA\in\Z^{n\times(n{+}\ell)}$} and finitely many zeroes in $(\C^\times)^n$. 
  Necessarily the exponent vectors span a full rank sublattice $\Z\calA$ of $\Z^n$.
  Let $\varphi_{\calA}$ be the map~$\eqref{Eq:varphi}$ and $C\in\R^{n\times(n+\ell)}$ be the 
  coefficient matrix of $F$.
  The kernel of $\varphi_\calA$ is the Pontryagin dual to $\Z^n/\Z\calA$.

  A Gale system dual to $F(x)=b$ is given by a choice of a basis 
  $\calB=\{\beta^{(1)},\dotsc,\beta^{(\ell)}\}\subset\Z^{n{+}\ell}$ for the integer
  nullspace of $\calA$, a choice of a basis $L:=\{L_1,\dotsc,L_{n{+}\ell}\}$ for the kernel of 
  $C^T$, and a choice of a point $\lambda\in\R^{n{+}\ell}$ with $C^T\lambda=b$.
  Considering each $L_i$ as a linear map on~$\R^{\ell}$, the system of master functions Gale
  dual to the Laurent system $F(x)=b$ is
 \begin{equation}\label{Eq:PGD}
   \prod_{i=1}^{n{+}\ell} (L_i(y)+\lambda_i)^{\beta^{(j)}_i}\ =\ 1
   \qquad j=1,\dotsc,\ell\,.
 \end{equation}
  When $\Z\calA$ has odd index in $\Z^n$, 
  the solutions to $F(x)=b$ in $(\R^\times)^n$ are scheme-theoretically isomorphic to the solutions
  to~$\eqref{Eq:PGD}$ in the complement of the hyperplane arrangement defined by 
  $L_i(y)=-\lambda_i$ for $i=1,\dotsc,n{+}\ell$.
  For any $\calA$ with $\Z\calA$ of full rank, the solutions to $F(x)=b$ in $(\R{_{>0}})^n$
  correspond to solutions to~$\eqref{Eq:PGD}$ where $L_i(y)>-\lambda_i$, 
for $i=1,\dotsc,n{+}\ell$. 
  These functions~\eqref{Eq:PGD} form a regular sequence in the hyperplane complement.
\end{proposition}

The effect of the choices of $\calB$, $L$, and $\lambda$ on the Gale system and its solutions
may be deduced from Proposition~\ref{P:Gale}.
\begin{corollary}
\label{C:basicChange}
 Changing the basis $\calB$ of the integer nullspace of $\calA$ will change the Gale system, but
 not its solutions in $\R^{\ell}$ or the hyperplane arrangement.
 Changing the basis $L$ and the point $\lambda$ has the effect of an affine-linear change of
 coordinates in $\R^\ell$. 
\end{corollary}

\begin{remark}
 There are two other forms for the Gale system.
 First, we may take logarithms of the equations~\eqref{Eq:PGD} to obtain the logarithmic form,
 \begin{equation}\label{Eq:log-form}
   \defcolor{\phi_j(y)}\ :=\ 
   \sum_{i=1}^{n{+}\ell} \log(L_i(y)+\lambda_i)\ =\ 0
   \qquad j=1,\dotsc,\ell\,.
 \end{equation}
 This only makes sense in the \demph{positive chamber} of the hyperplane
 complement, 
\[
   \defcolor{\Delta}\ :=\ 
    \{ y\in\R^\ell\mid L_i(y)>-\lambda_i\mbox{\ for all }i=1,\dotsc,n{+}\ell\}\,.
\]
 To obtain the other form, set $a^+:=\max\{a,0\}$ and $a^-:=\max\{-a,0\}$, for $a\in\Z$.
 Note that $a=a^+-a^-$.
 With this notation, we may clear the denominators of~\eqref{Eq:PGD} to get a 
 polynomial form of the Gale dual system,
 \begin{equation}\label{Eq:poly-form}
   \prod_{i=1}^{n{+}\ell} (L_i(y)+\lambda_i)^{(\beta^{(j)}_i)^+}\ -\ 
   \prod_{i=1}^{n{+}\ell} (L_i(y)+\lambda_i)^{(\beta^{(j)}_i)^-}\ =\ 0
   \qquad j=1,\dotsc,\ell\,.
 \end{equation}

\end{remark}

\begin{remark}\label{R:Bounded}
It is no loss of generality to assume that this positive chamber $\Delta$ is
bounded. 
Indeed, suppose that $\Delta$ is unbounded.
This occurs if and only if the coefficient matrix $C$ annihilates a positive vector.
Since $L_1,\dotsc,L_{n+\ell}$ span the dual space to $\R^\ell$, $\Delta$ is strictly convex and
therefore has a bounded face, $P$.
Then there is an affine form, $a_0 + L(y)$ where~$L$ 
is a linear function on $\R^\ell$ and $a_0>0$,
that is strictly positive on $\Delta$ and such that $P$ is
the set of points of $\Delta$ where this polynomial achieves its minimum value on $\Delta$.
Dividing by $a_0$, we may assume that the constant term is 1.

Consider the projective coordinate change $y\mapsto \oy$, where 
 \begin{equation}\label{Eq:PCC}
   \defcolor{\oy_j}\ :=\ y_j\cdot\frac{1}{1+L(y)}\qquad\mbox{for }j=1,\dotsc,\ell\,.
 \end{equation}
We have 
\[
   y_j\ =\ \oy_j\cdot\frac{1}{1-L(\oy)}\qquad\mbox{for }j=1,\dotsc,\ell\,.
\]
Note that $(1+L(y))(1-L(\oy))=1$.
  
The coordinate change~\eqref{Eq:PCC} manifests itself on affine forms as follows. 
If $p(y)$ is a affine form, then
 \begin{equation}
   p(y)\ =\ \op(\oy)\cdot\frac{1}{1-L(\oy)}
   \qquad\mbox{and}\qquad
   \op(\oy)\ =\ p(y)\cdot\frac{1}{1+L(y)}\,,
 \end{equation}
where $\defcolor{\op(\oy)}:=p(\oy)-p(0)L(\oy)$.
If $\Lambda_i(y):=L_i(y)+\lambda_i$ is one of the affine forms defining~$\Delta$, then 
$\overline{\Lambda_i}(\oy)=L_i(\oy)+ \lambda_i(1-L(\oy))$.
Under the projective transformation~\eqref{Eq:PCC}, the polyhedron~$\Delta$ is transformed into
$\oD$, where
\[
   \defcolor{\oD}\ :=\ \{\oy\in\R^\ell\mid 
    1-L(\oy)>0\ \mbox{ and }\ \overline{\Lambda_i}(\oy)>0\quad\mbox{for }i=1,\dotsc,n{+}\ell\}\,.
\]
Set $\defcolor{\overline{\Lambda_0}(\oy)}:=1-L(\oy)$ and
$\defcolor{\beta_{0}^{(j)}}:=-\sum_{i=1}^{\ell+n}\beta_{i}^{(j)}$.

\begin{proposition}[\cite{BS10}, Prop.~2.3]
\label{Prop:proj}
  Under the coordinate change~\eqref{Eq:PCC}, the system of master functions~\eqref{Eq:PGD} on
  $\Delta$ is transformed into the system
\[
   \prod_{i=0}^{\ell+n}  \overline{\Lambda_i}(\oy)^{\beta^{(j)}_i}\ =\ 1
    \qquad\mbox{for}\quad j=1,\dotsc,\ell\,,
\]
 on the bounded polyhedron $\oD$.
\end{proposition}

This coordinate transformation $y\mapsto \overline{y}$ is a concrete and
explicit projective transformation sending the unbounded polyhedron $\Delta$ to a
bounded polyhedron $\overline{\Delta}$.
We will henceforth assume that $\Delta$ is bounded.
\end{remark}

%
\subsection{Khovanskii-Rolle continuation}\label{SS:KhRo}
%
Khovanskii-Rolle continuation is a numerical algorithm introduced in~\cite{BS10} to compute the
solutions to a Gale system in the positive chamber.
It is based on the Khovanskii-Rolle Theorem and the proof of the fewnomial bound~\cite{BS} 
which bounds the number of solutions to a Gale system.

We first state the Khovanskii-Rolle Theorem.
Write \defcolor{$\calV_\Delta(f_1,\dotsc,f_k)$} for the common zeroes in $\Delta$ of 
functions $f_1,\dotsc,f_k$ defined on $\Delta$, counted with multiplicity and 
\defcolor{$J(f_1,\dotsc,f_\ell)$} for the Jacobian (determinant), $\det(\partial f_i/\partial y_j)$, of
$f_1,\dotsc,f_\ell$. 
For a curve $\gamma$ in $\Delta$, let \defcolor{$\ubc_\Delta(\gamma)$} be the number of unbounded
components of $\gamma$ in $\Delta$ (those that meet its boundary).
This is one-half the number of points of intersection of $\gamma$ with the boundary of $\Delta$.

\begin{theorem}[Khovanskii-Rolle~\cite{Kh91}]\label{T:Kh-Ro}
  Let $g_1,\dotsc,g_\ell$ be smooth functions defined on a polyhedral domain
  $\Delta\subset\R^\ell$ 
  with finitely many common zeroes and suppose
  $\defcolor{\gamma}:=\calV_\Delta(g_1,\dotsc,g_{\ell-1})$ is a smooth curve. 
  Let $\defcolor{J}$ be the Jacobian  of $g_1,\dotsc,g_\ell$.
  Then 
 \begin{equation}\label{Eq:KhRo}
   |\calV_\Delta (g_1,\dotsc,g_\ell)|\ \leq\ \ubc_\Delta(\gamma)\ +\
   |\calV_\Delta(g_1,\dotsc,g_{\ell-1},\,J)| \,.    
 \end{equation}
\end{theorem}

The main idea is that along an arc of the curve $\gamma$, the Jacobian vanishes at least once
between any two consecutive zeroes of $g_\ell$.
 \[
  \begin{picture}(262,130)(-20,0)
   \put(0,0){\includegraphics[height=120pt]{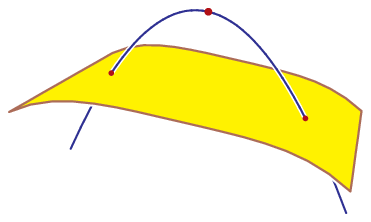}}
    \put(75,3){$\gamma=\calV_\Delta(g_1,\dotsc,g_{\ell-1})$}
    \put(-20,65){$g_\ell=0$}
    \put(148,120){$\calV_\Delta(g_1,\dotsc,g_{\ell-1},J)$}
    \put(146,122){\vector(-3,-1){30}}
  \end{picture}
 \]
If we let $S$ be the points of $\gamma$ where $J=0$ and $T$ be the points where $\gamma$ meets
the boundary of $\Delta$, then every zero of $g_\ell$ on $\gamma$ lies in a unique interval of
$\gamma\smallsetminus(S\cup T)$.
Thus, the Khovanskii-Rolle Theorem leads to the following continuation algorithm.

\begin{alg}\label{Alg:KRC_small}
 Using arclength continuation to follow $\gamma$ starting at points of $S\cup T$,
 where~$\gamma$ is traced in both directions from points of $S$ and is traced into the
 interior of the polyhedron~$\Delta$ from points of $T$, will recover all points of $\gamma$ where
 $g_\ell=0$, each at least twice.
\end{alg}

We apply this to a Gale system.
Let $\phi_1,\dotsc,\phi_\ell$ be the functions in the logarithmic form of a Gale
system~\eqref{Eq:log-form}.
Set $\gamma_\ell:=\calV(\phi_1,\dotsc,\phi_{\ell-1})$.
Then $\gamma_\ell$ is an algebraic curve in $\Delta$, as logarithmic functions $\phi_i$
define the same sets as the polynomial functions~\eqref{Eq:poly-form}, and these form a complete intersection.
Let $J_\ell$ be the Jacobian determinant of $\phi_1, \dotsc, \phi_\ell$.
Given the points $\defcolor{S_{\ell-1}}:=\calV(\phi_1,\dotsc,\phi_{\ell-1},J_\ell)$ 
where $J_\ell$ vanishes on $\gamma_\ell$ and the points $\defcolor{T_\ell}$ where $\gamma_\ell$
meets the boundary of $\Delta$, we may compute the solutions
$\defcolor{S_\ell}:=\calV_\Delta(\phi_1,\dotsc,\phi_\ell)$ using arclength
continuation along $\gamma_\ell$, by Algorithm~\ref{Alg:KRC_small}.

A genericity assumption on the coefficient matrix $C$ and the exponents $\calB$ implies that~$T_\ell$ is a subset
of the vertices of $\Delta$. 
By Lemma~3.4 (1) of~\cite{BS}, $\defcolor{\widetilde{J}_\ell}:=J_\ell\cdot\prod_i(L_i(y)+\lambda_i)$ is a
polynomial of degree $n$. 
Thus we may compute the solutions $S_\ell$ to the
Gale system $G$, given knowledge of the set $T_\ell$ and the set $S_{\ell-1}$ in which the
transcendental function $\phi_\ell$ defining $S_\ell$ is replaced by the low degree polynomial
$\widetilde{J}_\ell$. 
The Khovanskii-Rolle algorithm~~\cite{BS10} and fewnomial bound~\cite{BS} iterate this procedure.
Before explaining them, we discuss our genericity assumptions.

The exponent vectors $\calB=\{\beta^{(1)},\dotsc,\beta^{(\ell)}\}$ are the columns of a
$(n{+}\ell{+}1)\times\ell$-matrix, also written $\calB$, whose rows correspond to the affine functions defining
$\Delta$. 
For a face $P$ of $\Delta$ and any $j=1,\dotsc,\ell$, let \defcolor{$\calB_{P,j}$} be the submatrix of $\calB$
consisting of the first $j$ columns of the rows of $\calB$ corresponding to the affine forms vanishing on $P$.
We assumed $\Delta$ is bounded, which may be achieved by a projective change of coordinates, as
explained in Remark~\ref{R:Bounded}.

\begin{assumption}\label{A:one}
  We will make the following assumptions.
 \begin{enumerate}
  \item  The affine forms $L_i(y)+\lambda_i$ for $i=1,\dotsc,n{+}\ell{+}1$ are in general position in that 
          exactly $\ell$ of them vanish at each vertex of $\Delta$.
  \item  For every face $P$ of $\Delta$, the square matrix $\calB_{P,\codim(P)}$ annihilates no nonzero nonnegative
    row vector. 
  \item The sets $\gamma_j$ defined below are curves.
 \end{enumerate}
\end{assumption}

Assumption (1) is implied by the assumption that the coefficients of
the system $F(x)=b$ are general, which is mild, as we already assumed that the system had 
only isolated zeroes.
A consequence  is that $\Delta$ is a simple polyhedron.
That is, any face $P$ of codimension $r$ lies on exactly $r$ facets, and thus exactly $r$ affine forms
$L_i(y)+\lambda_i$ vanish on $P$. 

Assumptions (2) and (3) depend upon the exponent vectors $\calA$ of the monomials in $F$, as well
as the polyhedron $\Delta$ (which comes from $C$ and $b$), and therefore it is more
difficult for them to hold.
As explained in~\cite{BS}, these assumptions hold on a Zariski open subset of the spaces of matrices $C$ and
$\calB$ (when $\calB$ is allowed to have real entries).

We sketch the steps of Khovanskii-Rolle continuation under these assumptions.
For each $j=\ell,\ell{-}1,\dotsc,1$, set
$\defcolor{J_j}$ to be the Jacobian of $\phi_1.\dotsc,\phi_j,J_{j+1},\dotsc,J_\ell$
and define $\defcolor{\gamma_j}$ to be
$\calV_\Delta(\phi_1,\dotsc,\phi_{j-1}\,,\,J_{j{+}1},\dotsc,J_\ell)$.
By Assumption~\ref{A:one}(3), this is a curve.
Let $\defcolor{T_j}$ be the points where $\gamma_j$ meets $\partial\Delta$. 
Finally, for $j=0,\dotsc,\ell$ set 
$\defcolor{S_j}:=\calV_\Delta(\phi_1,\dotsc,\phi_j,J_{j+1},\dotsc,J_\ell)$.
Then $S_\ell$ is the set of solutions to the Gale system.

By Algorithm~\ref{Alg:KRC_small}, arclength continuation along $\gamma_j$ beginning from points
of $T_j$ and $S_{j-1}$ will compute all points of $S_j$.
Thus, the solutions $S_\ell$ of the Gale system in $\Delta$ may be computed iteratively
from $S_0$ and the sets $T_1,\dotsc,T_\ell$.
Computing these sets is feasible, for by Lemma~3.4 of~\cite{BS} and 
Corollary~\ref{C:faceCharacterization}, we have
 \begin{enumerate}
  \item $\widetilde{J_i}=J_i\cdot\Bigl(\prod_i(L_i(y)+\lambda_i)\Bigr)^{2^{\ell-j}}$ is a
      polynomial of degree $2^{\ell-j}\cdot n$. 

  \item The points of $T_j$ all lie on some faces of $\Delta$ of dimension $\ell{-}j$ and are the
    points of those faces where the polynomials 
    $\widetilde{J_j},\dotsc,\widetilde{J_\ell}$ vanish.
 \end{enumerate}

For $j = 2, \dotsc, \ell$, let $\defcolor{\mu_j} := \calV_\Delta(\phi_1, \dotsc, \phi_{j-1})$.  
The faces of $\Delta$ in (2) lie in the closure of the set $\mu_j$
and are determined by $\calB$ (Corollary~\ref{C:faceCharacterization} gives the precise statement).
This set $\mu_j$ is algebraic, as the transcendental functions $\phi_i$ define the same set in $\Delta$
as do the polynomials~\eqref{Eq:poly-form}. 
Furthermore, $\mu_j$ has dimension $\ell{-}j{+}1$.
Indeed, the rational function versions~\eqref{Eq:PGD} of the $\phi_i(y)$ form a regular sequence in the hyperplane
complement, so their common zero set $\mu_j$ has expected dimension $\ell{-}(j{-}1)$.
The assumption that $\gamma_j$ is a curve is that the Jacobians $J_{j{+}1},\dotsc,J_\ell$ form a complete
intersection on $\mu_j$.

%
\subsection{Our running example}\label{S:running_example}
Let us consider this for the example of Subsection~\ref{SS:Gale_algebra}.
The first two assumptions hold.
All polygons are simple, and no entry and no minor vanishes in the matrix of
exponents,
\[
   \calB\ =\ \left[\begin{array}{rrrrr}
      -1 & -3 &  1 & -1 & 3\\ 
       2 & -1 &  3 & -2 & -2 
     \end{array}\right]^T\,.
\]
The Jacobian of the logarithmic system~\eqref{Eq:Logs1} is the rational function,
\[
  J_2\ :=\   \frac{2x^3-16x^2y+12xy^2+6y^3 
      -\tfrac{31}{2}x^2+26xy-\tfrac{53}{2}y^2+\tfrac{9}{2}x+\tfrac{15}{2}y-2}%
    {xy(1-x-y)(\tfrac{1}{2}{-}x{+}y)(1+x-3y)}\ ,
\]
whose denominator is the product of the linear factors defining the lines in
Figure~\ref{F:Gale_pic1}.
Clearing the denominator and multiplying by 2 gives the cubic polynomial,
\[
   \rgbColor{0.8 0 0}{\widetilde{J}_2}\ :=\
   \rgbColor{0.8 0 0}{4x^3-32x^2y+24y^2x+12y^3-31x^2+52xy-53y^2+9x+15y-4}\,.
\]
The Jacobian $J_1$ of $g$ and $J_2$ has denominator the product of the squares of the forms
defining the lines, and its numerator is the sextic
 \begin{multline*}
  -56x^6+464x^5y-456x^4y^2-1792x^3y^3+1896x^2y^4+336xy^5-72y^6+640x^5-1960x^4y\\
   +456x^3y^2+3096x^2y^3-5176xy^4+480y^5-482x^4+2248x^3y-2416x^2y^2+4256xy^3+250y^4\\
   +61x^3-401x^2y-1769xy^2-491y^3-10x^2+664xy+278y^2-81x-101y+16\,.
 \end{multline*}
The system $J_1=J_2=0$ has $18$ solutions with $10$ real, but it has only two solutions~$S_0$
in the pentagon.
The curve $\rgbColor{0.9 0 0}{\gamma_1}$ is defined by the vanishing of the
Jacobian $\rgbColor{0.9 0 0}{J_2}$.
The set $T_1$ of points where it meets the boundary of the pentagon
consists of two points.

In Figure~\ref{F:KRC}, we illustrate the configuration of the sets $S_0,S_1,S_2,T_1$, and
$T_2$ in the pentagon of Figure~\ref{F:Gale_pic1}, along with the curves defined by $J_1$,
$J_2$ (written $\gamma_1$), $g$ ($\gamma_2$), and $f$.
\begin{figure}[htb]
  \begin{picture}(130,95)
   \put(0,0){\includegraphics[height=95pt]{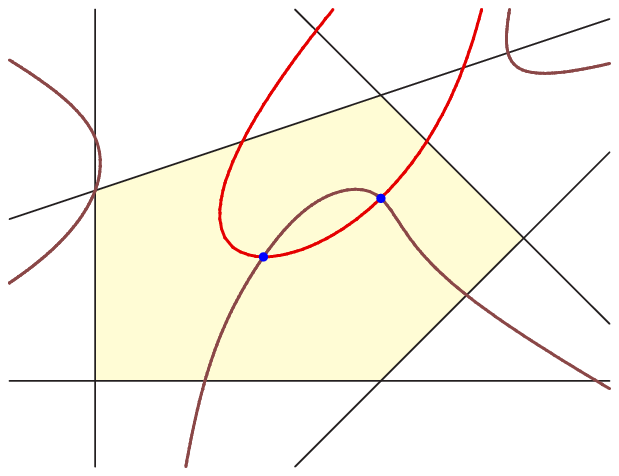}}
   \put(43,83){$\rgbColor{0.9 0 0}{\gamma_1}$}
   \put(2,30){$\rgbColor{0.54 0.28 0.28}{J_1}$}
   \put(25,3){$\rgbColor{0.54 0.28 0.28}{J_1}$}
   \put(112,71){$\rgbColor{0.54 0.28 0.28}{J_1}$}
    \put(65,23){$\Blue{S_0}$}
    \put(72.5,34){\vector(1,4){4.8}}
    \put(68,34){\vector(-3,2){12.5}}
  \end{picture}\qquad
  \begin{picture}(130,95)
   \put(0,0){\includegraphics[height=95pt]{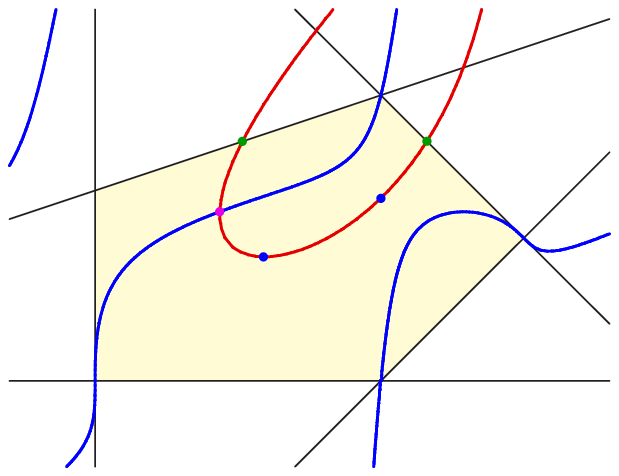}}
   \put(94,73){$\rgbColor{0.9 0 0}{\gamma_1}$}
   \put(9,85){$\Blue{g}$}   \put(6,3){$\Blue{g}$}
   \put(82,90){$\Blue{g}$}   \put(79,3){$\Blue{g}$}

    \put(65,23){$\Blue{S_0}$}
    \put(72.5,34){\vector(1,4){4.8}}
    \put(68,34){\vector(-3,2){12.5}}

    \put(34.5,23){$S_1$}  \put(39.5,34){\vector(1,4){4.2}}

    \put(23,81){$T_1$}
    \put(35.5,85.4){\vector(3,-1){49}}
    \put(35.5,81.8){\vector(1,-1){12}}
  \end{picture}\qquad
  \begin{picture}(130,95)
   \put(0,0){\includegraphics[height=95pt]{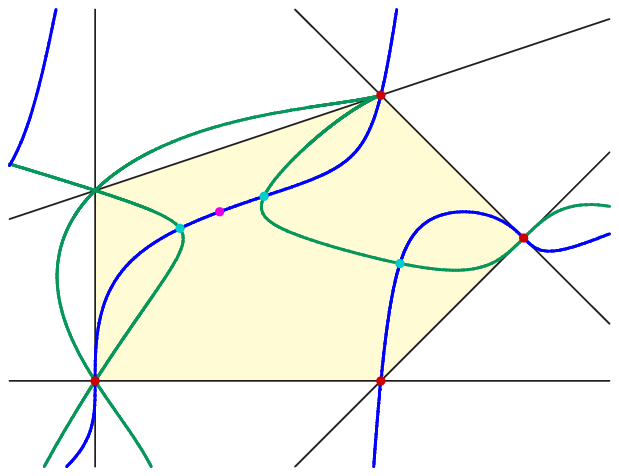}}

    \put(2,38){$\ForestGreen{f}$} \put(58,80){$\ForestGreen{f}$}
    \put(126,56){$\ForestGreen{f}$} 

    \put(9,85){$\Blue{g}$}    \put(125,44){$\Blue{g}$}
    \put(82,90){$\Blue{g}$}   \put(79,3){$\Blue{g}$}
    \put(43,23){$S_2$}
     \put(44,34){\vector(-1,2){6}}
     \put(48,34){\vector(1,4){4.5}}
     \put(52,32.5){\vector(3,1){26}}

    \put(32.5,82){$S_1$}  \put(38.2,79){\vector(1,-4){5.5}}

    \put(44,7){$T_2$}
    \put(41.5,12){\vector(-4,1){20}}  \put(54.5,12){\vector(4,1){20}} 

    \put(104,73){$T_2$}
    \put(107.7,70){\vector(0,-1){19.5}}  \put(102,77.8){\vector(-1,0){20}} 

  \end{picture}

\caption{Solving the Gale system with Khovanskii-Rolle continuation}
\label{F:KRC}
\end{figure}
The middle picture shows that arclength continuation along $\gamma_1$ in both directions
from each point of~$S_0$ and into the pentagon from each point of $T_1$ will find the point $S_1$
twice (the other four continuations terminate when they either reach a point of $S_0$
or exit the pentagon).
The picture on the right shows that arclength continuation along $\gamma_2$ in both directions
from the point of $S_1$ and into the pentagon from the four points of $T_2$ will reach the 
three points of~$S_2$, twice each.
These points are
 \begin{equation}\label{Eq:Gale_Sols}
  (0.29557, 0.32316)\,,\ 
  (0.14846, 0.26681)\,,\ 
  (0.53346, 0.20563)\,.
 \end{equation}
%

\begin{remark}\label{R:singular}
 In Figure~\ref{F:KRC}, we see that some of the points  $T_2$ are singular points of
 $f=0$, but not of $g=0$.
 In general the points $T_1$ are smooth points on the curve $\gamma_1$, but the other points
 $T_j$ with $j>1$ may be singular points of the curves $\gamma_j$ they lie upon.
 In Section~\ref{S:bound}, we will discuss when it is possible to eliminate 
 tracking from these possibly singular points.
 A future publication~\cite{BHNS} detailing methods (and software) will explore strategies for arclength  
 continuation along $\gamma_j$ from these singular points of the boundary.
 \end{remark}

\begin{remark}
 Having computed approximations to the points~\eqref{Eq:Gale_Sols}, we do not yet have
 solutions to the original polynomial system.
 By Gale duality, if \defcolor{$S_G$} is the set of solutions to the Gale system, then 
 $\varphi_\calA^{-1}(\psi(S_G))$ is the set of solutions to the original polynomial system.  
 Suppose $S_G^*$ is a set of approximations to solutions of the Gale system.
 The difficulty is that points of $\psi(S_G^*)$ likely do not lie in $\varphi_\calA((\R^\times)^n)$, and so 
 $\varphi_\calA^{-1}(\psi(S_G^*))$ is undefined.
 In Subsection~\ref{SS:Algo_sols}, we give a method to overcome this obstruction.
\end{remark}

%
\section{A Descartes' bound for fewnomials}\label{S:bound}
%

We give a refinement to the upper bound of $\frac{e^2+3}{4}2^{\binom{\ell}{2}}n^\ell$ for the
number of solutions to a Gale system that was discovered while developing our software.
This refinement involves sign conditions and gives a Descartes'-like bound for fewnomials.
We then discuss how this refinement affects the running of the Khovanskii-Rolle algorithm.

\begin{theorem}\label{T:newBound}
 Suppose that $\Phi(x)=0$ is a system of $n$ Laurent polynomials in $n$ variables involving
 $n{+}\ell{+}1$ monomials.
 If the exponent vector of one monomial lies in the relative interior of the convex hull of the exponent
 vectors of the other monomials and if the coefficient matrix of these other monomials does not have
 a positive vector in its kernel, then this system has at most
\[
   \Bigl(1+ 2^{-\ell}\bigl(1+\tfrac{\ell}{n}\bigr)\Bigr)
    2^{\binom{\ell}{2}}n^\ell
\]
 nondegenerate positive solutions.
\end{theorem}

We will give a proof of Theorem~\ref{T:newBound} as part of our discussion on our improvement to the
fewnomial bound.
We first derive the fewnomial bound originally given in~\cite{BS}.

As explained at the end of Subsection~\ref{SS:KhRo}, the solutions $S_\ell$ to the Gale system are
found by first computing $S_0:=\calV_\Delta(\widetilde{J_1},\dotsc,\widetilde{J_\ell})$ and certain 
sets $T_1,\dotsc,T_\ell$ where $T_j$ consists of the points of the boundary of $\Delta$ that meet
the curve $\gamma_j$.
By Corollary~\ref{C:faceCharacterization} below, $T_j$ consists of the points of
certain faces of $\Delta$ of dimension $\ell{-}j$ where $\widetilde{J_{j+1}},\dotsc,\widetilde{J_\ell}$ vanish.
As $\widetilde{J_i}$ is a polynomial of degree $2^{\ell-i}n$, we may replace faces of $\Delta$ by
their affine span and use B\'ezout's Theorem to estimate the cardinality of these sets,
 \begin{equation}\label{Eq:Numbers}
  \begin{array}{rcl}
   |S_0| &\leq & {\displaystyle \prod_{i=1}^\ell \deg(\widetilde{J_i})\ =\ 2^{\binom{\ell}{2}}n^\ell}\\
   |T_j| &\leq & {\displaystyle f_{\ell-j}(\Delta) \cdot
         \prod_{i=j+1}^\ell \deg(\widetilde{J_i})\ =\ f_{\ell-j}(\Delta) 2^{\binom{\ell-j}{2}}n^{\ell-j}\ ,}
  \end{array}
 \end{equation}
 where $f_i(\Delta)$ is the number of $i$-dimensional faces of $\Delta$.
 In fact, there is a tighter estimate for $|T_j|$ given below~\eqref{Eq:tighter}.

 From the Khovanskii-Rolle Theorem, we have 
 $|S_i|\leq |S_{i-1}|+\frac{1}{2}|T_i|$, and so
 \begin{eqnarray}
  |S_\ell| &\leq& |S_0|\ +\ \tfrac{1}{2}\bigl(|T_1|+\dotsb+|T_\ell|\bigr)\nonumber\\
          &\leq& 2^{\binom{\ell}{2}}n^\ell\ +\ 
           \frac{1}{2}\sum_{j=1}^\ell f_{\ell-j}(\Delta)2^{\binom{\ell-j}{2}}n^{\ell-j} \label{Eq:FBoundNoSlack} \\
          &<&2^{\binom{\ell}{2}}n^\ell\ +\ 
           \frac{1}{2}\sum_{j=1}^\ell
           \binom{n{+}\ell{+}1}{j}2^{\binom{\ell-j}{2}}n^{\ell-j}\nonumber\\
          &\leq& 2^{\binom{\ell}{2}}n^\ell\ +\ 
           \frac{1}{2}\sum_{j=1}^\ell \frac{2^{j-1}}{j!} 2^{\binom{\ell}{2}}n^\ell
          \ \ <\ \ \frac{e^2{+}3}{4}2^{\binom{\ell}{2}}n^\ell\,. \nonumber
 \end{eqnarray}
 The estimate of the last row is Lemma~3.5 in~\cite{BS}.

 Any difference in the two sides of any inequality (its \demph{slack}) propagates through
 to the final bound.
 One clear way that the estimates~\eqref{Eq:Numbers} could have slack would be if a system
 $\widetilde{J_{j+1}},\dotsc,\widetilde{J_\ell}$ on the complex affine span of a face did
 not have all of its solutions lying in that face.
 In practice, we have found that very few of the solutions lie on the desired face of $\Delta$.
 While this seems impossible to predict or control, the Khovanskii-Rolle continuation algorithm 
 always takes advantage of this source of slack.

 Another source of slack is our estimates for $|T_j|$, which depend on the number
 of certain faces of $\Delta$, as not all faces of codimension $j$ can contain points of $T_j$.
 We investigate the location of the points $T_j$.
For each $j=2,\dotsc,\ell$, we defined $\mu_j\subset\Delta$ to be closure of the points in the interior of
$\Delta$ given by
 \begin{equation}
   \phi_1(y)\ =\ \dotsb\ =\ \phi_{j-1}(y)\ =\ 0\,.
 \end{equation}
This set is algebraic and has dimension $\ell{-}j{+}1$.

\begin{lemma}\label{L:Why_Assumption}
 Suppose that Assumption~$\ref{A:one}$ $(1)$ holds.
 If $P$ is a face of $\Delta$ whose interior meets $\mu_j$, then 
 the submatrix $\calB_{P,j-1}$ of $\calB$ consisting of the first $j{-}1$ columns of rows corresponding to the 
 facets of $\Delta$ containing $P$ annihilates a nonzero nonnegative row vector.
 When this occurs, we have  $P\subset \mu_j$.
\end{lemma}

By annihilating a nonzero nonegative row vector, we mean there is a nonzero vector $v=(v_1,\dotsc,v_r)$ with
$v_i\geq 0$ for all $i$ such that $v \calB_{P,j-1}=0$.
Here, $r$ is the number of facets containing $P$.

\begin{proof}
 Let $p$ be a point lying in the interior of $P$.
 After an affine change of coordinates and a reordering, we may assume that $p=0$ and the
 affine forms defining $\Delta$ are
\[
    y_1\,,\ y_2\,,\ \dotsc\,,\ y_r\ ,\  
    \Lambda_{r+1}(y)\,,\ \dotsc\,,\ \Lambda_{n+\ell+1}(y)\,,
\]
 where $\Lambda_a(0)>0$ for $a>r$.
 That is, points of $P$ have their first $r$ coordinates equal to 0.

 Then $\mu_j$ is defined on the interior of $\Delta$ by the rational functions
\[
   y_1^{\beta_1^{(i)}}\dotsb \,y_r^{\beta_r^{(i)}}\cdot
   \prod_{a=r+1}^{n+\ell+1} \Lambda_a(y)^{\beta_a^{(i)}}\ =\ 1
   \qquad\mbox{for}\quad i=1,\dotsc,j{-}1\,.
\]
 After an analytic change of coordinates in the neighborhood of the origin $p$, we may assume that 
 these become monomial,
\[
   y_1^{\beta_1^{(i)}}\dotsb\, y_r^{\beta_r^{(i)}}\ =\ 1
   \qquad\mbox{for}\quad i=1,\dotsc,j{-}1\,.
\]
 The solution set meets the origin $p$ if and only if it contains a monomial curve of the form
\[
   (t^{\alpha_1}\,,\, t^{\alpha_2}\,,\,\dotsc\,,\, t^{\alpha_r})
\]
 with each exponent $\alpha_i$ nonnegative and not all are zero.
 But then $\alpha=(\alpha_1,\dotsc,\alpha_r)$ satisfies
\[
   \sum_{t=1}^r \beta_t(i)\alpha_t\ =\ 0\qquad
   \mbox{ for } i=1,\dotsc,j{-}1\,,
\]
so that $\alpha$ is a nonzero nonnegative row vector annihilated by $\calB_{P,j-1}$.

 The last statement follows as these arguments may be reversed.
\end{proof}

\begin{corollary}\label{C:faceCharacterization}
 Suppose that Assumption~$\ref{A:one}$ holds.
 Then $\mu_j\cap\partial\Delta$ is the union of those codimension $j$ faces $P$ of $\Delta$ where 
 $\calB_{P,j-1}$ annihilates a nonzero nonnegative row vector.
\end{corollary}

The points $T_j$ are the points of $\mu_j\cap\partial\Delta$ where the polynomials 
$\widetilde{J}_{j+1},\dotsc,\widetilde{J}_\ell$ vanish.
By Corollary~\ref{C:faceCharacterization}, this is the restriction of those $\ell{-}j$ polynomials to a
collection of faces, each of dimension $\ell{-}j$.

\begin{proof}
 Since $\mu_j$ has dimension $\ell{-}j{+}1$, its intersection with the boundary of $\Delta$ has
 dimension $\ell{-}j$.
 If this intersection meets the relative interior of a face $P$ of $\Delta$ of codimension $r$, then the matrix 
 $\calB_{P,j-1}$ annihilates a nonzero nonnegative row vector, by Lemma~\ref{L:Why_Assumption}.

 If $r<j$, then $\calB_{P,r}$ is a submatrix of $\calB_{P,j-1}$, and so $\calB_{P,r}$ annihilates a
 nonzero nonnegative row vector and contradicts our assumption.
 Thus $r\geq j$, and we see that the rest of the corollary follows from Lemma~\ref{L:Why_Assumption}. 
\end{proof}

Corollary~\ref{C:faceCharacterization} restricts the faces on which points of $T_j$ can lie and implies
the following tighter estimate,
 \begin{equation}\label{Eq:tighter}
    |T_j|\ \leq\ M_{\ell-j}(\Delta,\calB)\cdot 2^{\ell-j}n^{\ell-j}
    \qquad\mbox{for}\quad j>1\,,
 \end{equation}
where \defcolor{$M_{\ell-j}(\Delta,\calB)$} counts the faces $P$ of $\Delta$ with codimension
$j$ for which $\calB_{P,j-1}$ annihilates a nonzero nonnegative row vector.

\begin{proof}[Proof of Theorem~$\ref{T:newBound}$]
 Let $\Phi(x)=0$ be a system of $n$ Laurent polynomials in $n$ variables involving $n{+}\ell{+}1$
 monomials.
 Let $x^a$ be a monomial in $\Phi$ whose exponent vector $a$ lies in the interior of the
 convex hull of the other exponent vectors.
 We may replace $\Phi$ by $x^{-a}\Phi$ to obtain a system of Laurent polynomials with the same
 zeroes in $(\R^\times)^n$ having a constant term such that 0 lies in the interior of the convex hull of
 the exponent vectors of its remaining monomials.

 Rewrite $\Phi(x)=0$ as $F(x)=b$ where $F$ has $n{+}\ell$ monomials with exponent vectors $\calA$,
 and the origin lies in the interior of the convex hull of $\calA$.
 Thus the origin is a nonnegative integer combination of the vectors in $\calA$, with every vector in $\calA$
 participating in this expression.
 Consequently, this expression is a positive vector in the integer nullspace of $\calA$, and we may assume it is 
 the first column of the matrix $\calB$ of exponents for the Gale system.
 This implies that no $j\times (j{-}1)$ submatrix of the first $j{-}1$ columns of $\calB$ has a
 positive row vector in its kernel, and so $M_{\ell-j}(\Delta,\calB)=0$, when $j>1$.

 If the positive chamber $\Delta$ is bounded, then it has at most $n{+}\ell$ facets and 
 $T_2,\dotsc,T_\ell$ are empty.
 This gives the estimate
 \begin{eqnarray}\nonumber
   |S_\ell|&\leq& 
    |S_0|+\tfrac{1}{2}|T_1|\ \leq\ 
    2^{\binom{\ell}{2}}n^\ell + (n{+}\ell)2^{\binom{\ell-1}{2}-1}n^{\ell-1}\\
    &=&\Bigl(1+ 2^{-\ell}\bigl(1+\tfrac{\ell}{n}\bigr)\Bigr)
    2^{\binom{\ell}{2}}n^\ell\,. \label{Eq:smBd}
 \end{eqnarray}
 If $\Delta$ is unbounded, then, after the transformation of Remark~\ref{R:Bounded},
 $\Delta$ possibly has $n{+}\ell{+}1$ facets and the first column of $\calB$ has exactly one
 negative entry.

 To complete the proof, suppose further that the coefficient matrix $C$ of those
 other exponent vectors has no positive vector in its kernel.
 This is equivalent to $\Delta$ being bounded, and we have the bound~\eqref{Eq:smBd}.
\end{proof}

The sets $T_2,\dotsc,T_\ell$ are empty under the hypotheses of Theorem~\ref{T:newBound} and when the first
column of $\calB$ has every entry positive.
In this case, Assumption~\ref{A:one} (2) on $\calB$ is satisfied.

We noted in Remark~\ref{R:singular} that the points $T_j$ for $j\geq2$ will typically be singular
points of the curve $\gamma_j$ as $\mu_j$ is singular along the boundary of $\Delta$.
Minimizing the number of points~$T_j$ (or eliminating them altogether as in
Theorem~\ref{T:newBound}), reduces the difficult path tracking required to track from these singular
points $T_j$, thereby improving the performance of the Khovanskii-Rolle continuation algorithm.

%
\section{Algorithmic details for the Gale transform}\label{S:algos}
%

Obtaining the solutions of a fewnomial system by Khovanskii-Rolle continuation is a 3-step process.
 \begin{enumerate}
  \item Apply the Gale transform to a fewnomial system $F(x)=b$, producing a dual Gale system $G(y)=1$ of
    master functions. 
  \item Use Khovanskii-Rolle continuation to find approximations to elements of $S_G$ which are the real solutions
    of $G(y)=1$ in the positive chamber $\Delta \subset \mathbb{R}^\ell$.   
  These approximations form the set $S_G^*$.
  \item For each $s^* \in S_G^*$, transform $s^*$ to $t^*$, a numerical approximation to $t \in S_F$.  
  Set $S_F^* := \{ t^* \mid t^*\text{ is an approximation to } t \in S_F \}$.
 \end{enumerate}
A fewnomial system may be solved using Gale duality by other root-finding methods in place of  
Khovanskii-Rolle continuation in Step 2.
We provide pseudocode for the Gale transform of Step 1 (Algorithm~\ref{alg:wrapping}) and for the
transformation of the approximate solutions $S_G^*$ of the Gale system $G$ to the approximate solutions $S_F^*$ of
the fewnomial system $F$ in Step 3 (Algorithm~\ref{alg:unwrapping}).  
Two steps of Algorithm~\ref{alg:wrapping} require extra attention; they are  
detailed in the final two subsections. 

%
\subsection{Converting a fewnomial system to a system of master functions}\label{SS:Algo_GD}
%

Using the notation of Section~\ref{S:KhRo}, we give pseudocode for the basic algorithm for
converting a system $F(x)=b$ of fewnomials into a system of master functions $G(y)=1$.

\begin{algorithm}\label{alg:wrapping}\mbox{\ }

\noindent
{\bf Input}:  $\ell, n$, an $n\times (n{+}\ell)$ integer matrix $\calA$, an
$n\times(n{+}\ell)$ coefficient matrix $C$, and target column vector $b$.

\noindent
{\bf Output}:  A system $G(y)=1$ of $\ell$ master functions in $\ell$ variables that is Gale dual to the system
$F(x)=b$. 

\begin{algorithmic}

\STATE 1) Choose $\ell$ monomials and reorder the columns of $C$ so that these monomials are the last $\ell$ columns.
\STATE 2) Put $[C\mid b]$ in echelon form, obtaining a matrix $[-I_n\mid L\mid\lambda]$.  
\STATE 3) From the rows of $L$ and entries of $\lambda$ form the affine forms $L_i(y)+\lambda_i$.
\STATE 4) Choose $\ell$ integral basis vectors $\beta_i$ of the nullspace of $\calA$ to construct
          \mbox{$\calB = \{\beta^{(1)}, \dotsc, \beta^{(\ell)}\}$}.
\STATE 5) Use $\calB$ and the forms $L_i(y)+\lambda_i$ to form the Gale dual system
  $G(y)=1$ as in~\eqref{Eq:PGD}. 

\end{algorithmic}
\end{algorithm}

We explain some of details in the design and implementation of each step of Algorithm~\ref{alg:wrapping}, and 
then illustrate it with an example.
 \begin{itemize}
 \item[1)]  The effect of this choice is an affine change of coordinates in $\R^\ell$, which affects the shape of
   the positive chamber $\Delta\subset\R^{\ell}$ and the performance of the Khovanskii-Rolle continuation when no
   scaling of the Gale system functions occur. 
   With a simple scaling routine, discussed in Section~\ref{S:software}, the parameterization of the monomials
   appears to not matter when $\ell=2$.

\item[2)]  We simply use Gaussian elimination with partial pivoting.

\item[3)]  If the space $\R^{n+\ell}$ has coordinates $z\in\R^n$ and $y\in\R^\ell$, the
      echelon form $[-I_n\mid L\mid\lambda]$ gives a parameterization of the set $C(z,y)=b$ by the
      affine forms $z_i=L_i(y)+\lambda_i$ for $i=1,\dotsc,n$.

\item[4)]  We use the LLL algorithm~\cite{LLL} to compute an initial basis $\calB$ and then
       find vectors in the row space with few negative entries, a heuristic to minimize 
       $M_{\ell-j}(\Delta,\calB)$~\eqref{Eq:tighter},  which is the number of codimension $j$ faces of $\Delta$ on
       which the points $T_j$ may lie. 
       This step is the focus of Section~\ref{S:null}.

\item[5)]  If we set $L_{n+j}(y):=y_j$ and $\lambda_{n+j}:=0$ for $j=1,\dotsc,\ell$, then the
  vectors in $\calB$ and affine functions $L_i(y)+\lambda_i$ for $i=1,\dotsc,n{+}\ell$ 
      together give the Gale system $G(y)=1$ as in~\eqref{Eq:PGD} in Proposition~\ref{P:Gale}.
\end{itemize}

\begin{example}\label{Ex:example_positive}
Consider the fewnomial system, where $n=5$ and $\ell=2$ 
\begin{equation}\label{Eq:Sept_Few}
 \begin{array}{rcl}
-a^{-1}b^2c^2d \ +\hspace{18pt}\left(\frac{1}{2} b^2c\ +\ 2b^{-4}c^{-7}d^{-5}e^{-1}\ -\ 1\right)&=&0\,,\rule{0pt}{14pt}\\
-ac          \ +\hspace{12.8pt}\frac{1}{2}(\frac{1}{4}b^2c\ -\ \ \: b^{-4}c^{-7}d^{-5}e^{-1}\ +\ 1)&=&0\,,\rule{0pt}{14pt}\\
-bc^4d^4     \ +\ \frac{1}{4}(-\frac{1}{4}b^2c\ -\ 3b^{-4}c^{-7}d^{-5}e^{-1}\ +\ 6)&=&0\,,\rule{0pt}{14pt}\\
-d           \ +\ \frac{1}{2}(-\frac{3}{4}b^2c\ -\ 2b^{-4}c^{-7}d^{-5}e^{-1}\ +\ 8)&=&0\,,\rule{0pt}{14pt}\\
-e           \ +\hspace{10.8pt}(-\frac{1}{2}b^2c\ +\ 2b^{-4}c^{-7}d^{-5}e^{-1}\ +\ 3) &=&0\,,\rule{0pt}{14pt}
 \end{array}
\end{equation}
with support 
$\{ a^{-1}b^2c^2d\,,\, ac\,,\, bc^4d^4\,,\, d\,,\, e\,,\, b^2c\,,\,b^{-4}c^{-7}d^{-5}e^{-1}\}$
and thus matrix of exponents
\[
 \calA\ =\  \left[
   \begin{array}{rrrrrrr}
   -1 & 1 & 0 & 0 & 0 & 0 &  0 \\
    2 & 0 & 1 & 0 & 0 & 2 & -4 \\
    2 & 1 & 4 & 0 & 0 & 1 & -7 \\
    1 & 0 & 4 & 1 & 0 & 0 & -5 \\
    0 & 0 & 0 & 0 & 1 & 0 & -1 \\
   \end{array}
  \right]
\]
and coefficient matrix
\[
  [C\mid b]\ =\ \left[
  \begin{array}{rrrrrcc|c}
    -1 & 0 & 0 & 0 & 0 &  1/2  &  2  & -1\\ 
     0 &-1 & 0 & 0 & 0 &  1/8  &-1/2 & 1/2\\
     0 & 0 &-1 & 0 & 0 & -1/16 &-3/4 & 3/2\\
     0 & 0 & 0 &-1 & 0 & -3/8  & -1  & 4\\
     0 & 0 & 0 & 0 &-1 & -1/2  &  2  & 3\\
  \end{array}
  \right]\,
\]
where the first seven columns correspond to the seven monomials in $\calA$
and the final column corresponds to the constants.  

We presented this fewnomial system with a coefficient matrix in echelon form for the given order of the
monomials, so Steps 1 and 2 of Algorithm~\ref{alg:wrapping} are complete.  
If we set $s:=b^2c$ and $t:=b^{-4}c^{-7}d^{-5}e^{-1}$, the rows of this echelon matrix express each of the first
five monomials as affine functions of the parameters $s$ and $t$, so that, for example,
$a^{-1}b^2c^2d = \frac{1}{2}s + 2t - 1$.  
We have the seven $(=5+2)$ affine forms,
 \begin{equation}\label{lins}
  \begin{array}{rcl}
   \Lambda_1(s,t) &=& \frac{1}{2}s+2t-1\rule{0pt}{14pt}\\
   \Lambda_2(s,t) &=& \frac{1}{8}s-\frac{1}{2}t+\frac{1}{2}\rule{0pt}{14pt}\\  
   \Lambda_3(s,t) &=& -\frac{1}{16}s - \frac{3}{4}t+\frac{3}{2}\rule{0pt}{14pt}\\
   \Lambda_4(s,t) &=& -\frac{3}{8}s-t+4\rule{0pt}{14pt}\\
   \Lambda_5(s,t) &=& -\frac{1}{2}s+2t+3\rule{0pt}{14pt}\\
   \Lambda_6(s,t) &=& s\rule{0pt}{14pt}\\
   \Lambda_7(s,t) &=& t\rule{0pt}{14pt}
  \end{array}
 \end{equation}

Given a vector $\beta=[\beta_1,\ldots,\beta_7]^T$ in the nullspace of $\calA$, we 
have $\prod_{i=1}^7(x^{a_i})^{\beta_i}=1$, where $a_i$ is the $i^{\rm th}$ exponent vector in $\calA$.
Under the substitution $x^{a_i}=\Lambda_i(s,t)$, we obtain 
$\prod_{i=1}^7\Lambda_i(s,t)^{\beta_i}=1$.
One choice of basis $\calB$ for the nullspace of $\calA$ is
\[
   \{ [4, 4, 2, 3, 3, 1, 3]^T\,,\ [-1, -1, 2, -2, 1, 2, 1]^T\}\,.
\]
This gives the system of master functions (written in the polynomial form~\eqref{Eq:poly-form})
 \begin{equation}
  \begin{array}{rcl}
     g(s,t) &=& (s{+}4t{-}2)^4(s{-}4t{+}4)^4(s{+}12t{-}24)^2(3s{+}8t{-}32)^3(s{-}4t{-}6)^3st^3\\
    &&\hskip3.5in-68719476736,\\\rule{0pt}{13pt}
   f(s,t) &=& -2(s{+}12t{-}24)^2(s{-}4t{-}6)s^2t\ -\ (s{+}4t{-}2)(s{-}4t{+}4)(3s{+}8t{-}32)^2\,.
 \end{array}
 \end{equation}
Figure~\ref{fig:masters} displays the algebraic curves in $\R^2$ defined by these polynomials,
together with the lines defined by the affine forms $\Lambda_i(s,t)=0$ for $i=1,\ldots,7$.
\begin{figure}[htb]
\begin{picture}(225,187)(-15,0)
  \put(0,0){\includegraphics[width=210pt]{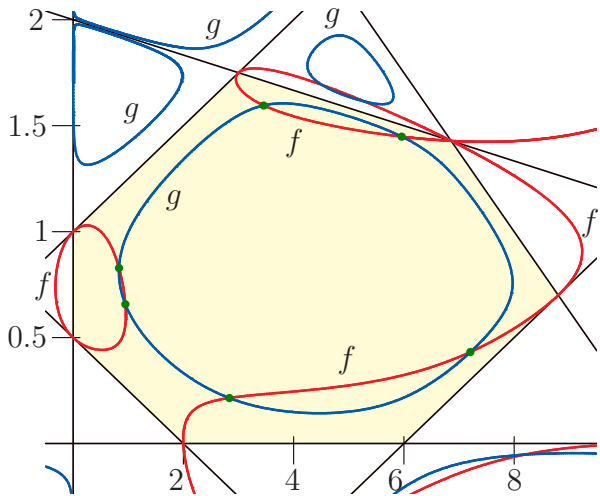}}
  \put(-15,56){$0.5$}  \put(-5,96){$1$} \put(-15,136){$1.5$}  \put(-7,176){$2$}
  \put(46,2){$2$}  \put(90,2){$4$} \put(129,2){$6$}  \put(174,2){$8$}
  \put(-5,76){$f$} \put(110,48){$f$} \put(90,130){$f$} \put(202,100){$f$}
  \put(29,143){$g$} \put(45,110){$g$} \put(60,175){$g$}\put(105,179){$g$}
 \end{picture}
\caption{Master function curves in the septagon}
\label{fig:masters}
\end{figure}

The original system~\eqref{Eq:Sept_Few} has 95 complex solutions with 11 real and six in the positive
orthant. 
The hypotheses of Theorem~\ref{T:newBound} hold; with $n=5$ and $\ell=2$, its improved bound is $87$ by
Theorem~\ref{T:newBound} and $67$ using~\eqref{Eq:FBoundNoSlack}. 
These hypotheses imply that the curve $g=0$ does not meet the boundary of the septagon $\Delta$.

\end{example}                                    

%
\subsection{Transforming Gale system solutions to fewnomial solutions}\label{SS:Algo_sols}
%

The output of Khovanskii-Rolle continuation is a set \defcolor{$S_G^*$} of approximations to solutions $S_G$ of the
Gale system in the positive chamber $\Delta$.  
Converting these points of $\Delta$ to points \defcolor{$S_F^*$} of $\R^n_>$ that are
approximations to solutions $S_F$ of the fewnomial system is problematic, for there is no
natural map between points of $\Delta$ and $\R^n_>$, except on the solutions to the two systems. 
We present a simple method to accomplish this conversion.

Let $s^* \in S_G^*$ and suppose that $s$ is the point of $S_G$ that $s^*$ approximates.
Applying the map $\psi$~\eqref{Eq:psi} of Subsection~\ref{SS:Geometry} whose components are the
affine functions $L_i(y)+\lambda_i$ gives points $\psi(s^*)$ and $\psi(s)$ in $\R^{n+\ell}_>$ that
lie on the affine plane $H$.
The point $\psi(s)$ lies on $\varphi_\calA(\R^n_>)$, but the point $\psi(s^*)$ 
does not lie on $\varphi_\calA(\R^n_>)$ (unless $s^*\in S_G$), as the former is a solution of the Gale system,
while the latter typically is not.  

To remedy this, choose a square submatrix $\calA'$ of rank $n$ of the matrix $\calA$ of exponents.
Projecting $\R^{n+\ell}$ to those coordinates indexed by the columns of $\calA'$ gives a map
\mbox{$\pi_{\calA'}\colon\R^{n+\ell}\to\R^n$}.
The composition
\[
   \R^n_>\ \xrightarrow{\ \varphi_{\calA}\ }\ 
   \R^{n+\ell}_>\ \xrightarrow{\ \pi_{\calA'}\ }\ \R^n_>
\]
is the map $\varphi_{\calA'}$, which is invertible on $\R^n_>$.
We set
\[
   \defcolor{t^*}\ :=\ \varphi_{\calA'}^{-1}\circ \pi_{\calA'}\circ \psi(s^*)\,.
\]
Since $t=\varphi_{\calA'}^{-1}\circ \pi_{\calA'}\circ \psi(s)=\varphi_{\calA}^{-1}\circ \psi(s)$
is the solution to the fewnomial system, and these maps are differentiable,
we expect that $t^*$ is close to $t$.
We formalize this transformation.

\begin{algorithm}\label{alg:unwrapping}\ 

\noindent
\mbox{{\bf Input:}} The exponent matrix $\calA$ and affine functions $L_i(y)+\lambda_i$ from 
       Algorithm~\ref{alg:wrapping}, and the set $S_G^*=\{s_i^* \mid i=1,...,N\}$ of
      approximations to solutions of the Gale system $G$ in the positive chamber $\Delta$.

\noindent
\mbox{{\bf Output:}}  A set $S_F^*=\{t^*_i\mid i=1,...,N\}$ of approximations to positive solutions of
the original fewnomial system $F$.

\begin{algorithmic}
\STATE  Choose a full rank $n\times n$ submatrix $\calA'$ of the matrix $\calA$ and compute
      $(\calA')^{-1}$. 

\FOR{$i := 1$ to $N$}
  \STATE a) Evaluate the map $\psi$ at the point $s^*_i$ to get the point $z^*_i\in\R^{n{+}\ell}_>$,
  \[
     z_i^* :=\ (L_1(s_i^*)+\lambda_1,\dotsc,L_{n+\ell}(s_i^*)+\lambda_{n+\ell})\,.
  \]
  \STATE b) Project $z_i^*$ to $\R^n_>$ using $\pi_{\calA'}$, set $v^*_i:=\pi_{\calA'}(z^*_i)$.
  \STATE c) Let $w^*_i\in\R^{n}$ be the coordinatewise logarithm of $v^*_i$.
  \STATE d) Set $x^*_i:= (w^*_i(\calA')^{-1})^T$ and let $t^*_i$ be the coordinatewise exponentiation
  of $x^*_i$.
  \ENDFOR

\end{algorithmic}
\end{algorithm}

Note that the homeomorphism $\varphi_{\calA'}\colon\R^n_>\to\R^n_>$ becomes the linear isomorphism
induced by $x\mapsto \calA'x$ in logarithmic coordinates.
Steps c) and d) invert this transformation.

Even if $s^*_i$ is certified to be an approximate solution to the Gale
system with associated solution $s_i$ (i.e.\ $s^*_i\to s_i$ under Newton
iterations, doubling the number of significant digits with each iteration), then it need not be the
case that $t^*_i$ is an approximate solution to the fewnomial system with associated solution
$t_i:= \varphi_{\calA}^{-1}\circ\psi(s_i)$.
However, we may check this using, for example, the software 
{\tt alphaCertified}~\cite{HS}.
If $t^*_i$ is not certified to be an approximate solution, then we could refine $s^*_i$ to be closer
to $s_i$ and reapply the transformation.
Our software has the functionality to do this, either producing a soft certificate (all calculations
done in floating point arithmetic) or a hard certificate (calculations done with rational
arithmetic), which is a proof that $t^*_i$ is an approximate solution.

\begin{example}
An approximation to a solution of $f(s,t)=g(s,t)=0$ from Example~\ref{Ex:example_positive}
lying in the positive chamber is $s^*_1=( 0.94884808, 0.65721633)$.
We follow the steps of Algorithm~\ref{alg:unwrapping} to recover a solution of the original fewnomial system.

We first choose an invertible submatrix of $\calA$.  The first $5 (= n)$ columns of $A$ suffice as these are linearly independent.  
Evaluating the map $\psi$ at $s^*_1$ gives the point 
\[
   z^*_1\ =\ \left[\begin{array}{r}
  0.78885671 \\
  0.28999784 \\
  0.94778474 \\
  2.98696564 \\
  3.84000863 \\
  0.94884808 \\
  0.65721633 \end{array}\right]\,.
\] 
 Projecting to the first $5$ coordinates of $z^*_1$ and taking coordinatewise logarithms yields
\[
  w^*_1\ =\  \left[\begin{array}{r}
 -0.23717058 \\ 
 -1.23788180 \\ 
 -0.05362787 \\ 
 1.09425803 \\ 
 1.34547461
 \end{array}\right]\,.
\] 
Proceeding to the next step, we evaluate 
\[ 
   x^*_1\ =\  (w^*_1(\mathcal{A}'^{-1}))^T\ =\ \left[\begin{array}{r}
  0.020520123 \\ 
  0.60294767 \\ 
  -1.25840192 \\ 
  1.09425803 \\ 
  1.34547461\end{array}\right]\,.
\]
Finally,
\[
  t^*_1\ =\ \exp(x^*_1)\ = \ 
  \left[\begin{array}{r}
  1.0207321 \\
  1.8274977 \\ 
  0.28410769 \\ 
  2.9869656 \\ 
  3.8400086\end{array}\right]\,.
\]
We carry out Algorithm~\ref{alg:unwrapping} with $s_1^{*}$ starting with $135$ bits of precision for the input and truncate the output to an approximation that can be stored in one $64$-bit floating point number.  

Using {\tt alphaCertified}~\cite{HS}, 
we perform two Newton iterations on this approximation, doubling the precision each time.
We obtain a soft-certification of this refined approximation, and then convert the numerical 
approximation to its rational form and obtain a hard certification that proves that $t_1^{*}$ converges to $t_1$.
\end{example}

%
\subsection{Heuristic for the nullspace basis}\label{S:null}

A choice in the Gale transform that affects the efficiency of numerical tracking in the
Khovanskii-Rolle algorithm is the selection of the basis $\calB$ of the annihilator of $\calA$.  

Observe that the points of $T_j$ for $j\geq 2$ are often singular points of the curve $\gamma_j$.
For efficiency and numerical stability, we want to avoid tracking near these points. 
Corollary~\ref{C:faceCharacterization} implies that the curve $\gamma_j$ may approach a codimension $j$
face $P$ of $\Delta$ only when $\calB_{P,j-1}$ annihilates a nonzero nonnegative row vector.
By Corollary~\ref{C:basicChange}, changing the basis elements $\beta^{(j)}$ of $\calB$ will change the Gale system
as well as the curves $\gamma_j$, without changing the solutions.
Ideally, we would use this freedom to choose the basis $\calB$ of the annihilator of $\calA$ to minimize the
number of faces of $P$ on which points $T_j$ for $j\geq 2$ could lie.

The number of possible choices for $\calB$ up to patterns for the signs of the coordinates suffers from a
combinatorial explosion~\cite{joseph60} as $\ell$ increases.  
Consequently, this optimization problem may become infeasible for $\ell$ large.
Instead, we find elements of the annihilator of $\calA$ that have the fewest number of negative coordinates.
We expect that this choice will tend to minimize the number of faces on which the points $T_j$ could lie.
In the case of Theorem~\ref{T:newBound}, this choice achieves the minimization.
When there is a positive vector annihilating $\calA$, we may choose $\calB$ to consist only of positive vectors,
and so no submatrix annihilates any nonzero nonnegative row vector.
By Lemma~\ref{L:Why_Assumption}, no face of $P$ will contain points of $T_j$. 

Once a basis for the annihiliator of $\calA$ is known, the algorithm proceeds differently depending upon the size
of $\ell$.
For small $\ell$, we can enumerate all possible patterns for the signs
of the coordinates of vectors in the annihilitor of $\calA$, and then choose a basis $\calB$ consistsing of
vectors having the fewest number of negative coordinates.  
When $\ell$ is large, we take random integer linear combinations of a basis of the annihilator of $\calA$ and
choose of these generated vectors those with the minimum number of negative entries.

This heuristic does not always work in parctice, as we observed in our test suite described below.
One problem is when a point of some $S_j$---a solution to an intermediate problem---lies close to
the boundary of $\Delta$.
Such solutions may be numerically ill-conditioned, and the subroutine in the Gale duality software package that
tracks the curves $\gamma_j$ near the boundary of $\Delta$ (the monomial tracker described
in~\cite[\S~4.3]{BS10}) may miss such ill-conditioned solutions.
A future publication~\cite{BHNS} will focus on the design and performance of the monomial tracker~subroutine.

%
\section{Software}\label{S:software}

The algorithms of Section~\ref{S:algos} are implemented in a software package, {\tt galeDuality}~\cite{gd}, which
is available at each author's website and written in {\tt C++}.
This will be the front end for a package we are developing for Khovanskii-Rolle continuation.
We describe a test suite~(\ref{S:testsuite}) to evaluate the proposed heuristic in Section~\ref{S:null}, 
and some implementation details of the package~(\ref{SS:software}).

\subsection{Test suite}
\label{S:testsuite}
As noted in Sections~\ref{S:KhRo} and~\ref{S:bound}, Khovanskii-Rolle continuation finds the solutions $S_\ell$
to a Gale system in a polyhedron $\Delta$ using arclength continuation along curves $\gamma_j$ starting at points
$\defcolor{T_0}:=S_0,T_1,\dotsc,T_\ell$ where $T_j$ are solutions to $\ell-j$ Jacobian determinants (having
B\'ezout number $2^{\binom{\ell-j}{2}}n^{\ell-j}$) in certain codimension $j$ faces of $\Delta$.
The set $T_0$ lies in the interior of the polyhedron $\Delta$ and points of $T_1$ could lie in every facet, but
the other $T_j$ lie in only a subset of the faces of codimension $j$.
The heuristic of Subsection~\ref{S:null} is intended to minimize or eliminate the number of faces contributing to
$T_j$ for $j\geq 2$.

To test this heuristic, we created a test
suite of Gale systems for each $(\ell, n)$ pair with $\ell=2$ and $n=2,\dotsc,13$.

For each $n$, we generated an $(n{+}2)$-gon by successively selecting random rational points in an annulus until we had 
$n{+}2$ points in convex position. These points formed the vertices of the $(n{+}2)$-gon.
The edges of this polygon then give $n{+}2$ affine forms with relatively prime integer coefficients.
To create a Gale system, we generated two vectors of exponents whose components correspond to the affine forms.
Both had integer components with absolute values selected uniformly at random in the range $1,\dotsc,10$.
One vector, $\beta^+$, had all components positive, but the other, $\beta^\pm$, had signs alternating as much as
possible traversing the polygon cyclically.  
This was intended to emulate the configuration of the master function curves in Figure~\ref{fig:masters}.

Our test suite consists of $100$ random Gale systems for each $n$.
For each random system, we conducted two types of runs of Khovanskii-Rolle continuation, one for each ordering of
the exponent vectors $\beta^+$ and $\beta^\pm$ as columns in $\calB$.
When $\beta^+$ is the first column of $\calB$, the set $T_2$ is empty, by Theorem~\ref{T:newBound}, but when $\beta^\pm$ is the first column,
$T_2$ consists of the vertices defined by affine forms such that the components of $\beta^\pm$ had differing signs.
Our proxy to test for the affect of the heuristic in Section~\ref{S:null} was to compare the average of $10$ running times for
these two different choices.

For any fixed $n$, in a head-to-head comparison of the average run-time of Khovanskii-Rolle continuation for the two different choices, 
choosing $\beta^+$ as the first column of $\calB$ outperformed $\beta^\pm$ for $>92\%$ of the systems in the test suite.  
For the majority of the instances in which the choice of $\beta^\pm$ outperformed $\beta^+$, 
the difference between the run times was insignificant.  

In about $15\%$ of all the systems in the test suite, the choice of $\beta^+$ in
the current implementation of the Khovanskii-Rolle continuation algorithm yielded incorrect results.
This occurs because some curve $\gamma_j$ is close to the boundary of $\Delta$.  
The monomial tracker subroutine fails to find some point $S_j$ when $S_j$ is on the curve $\gamma_j$ 
when $\gamma_j$ is being approximated as a monomial curve near the boundary of $\Delta$.  
As $n$ increased, we observed that this event occurred more frequently.
We believe this is due to the inadequacy of the monomial tracker in the proof-of-concept implementation released
with~\cite{BS10}, and improving that is a~future~project.

In addition, since the curves $\gamma_j$ are defined by the vanishing of a combination of Jacobian determinants and master functions, 
all these functions may have large coefficients, as seen in the examples we gave.
This affects the efficiency of the arclength continuation and computation of the points~$T_j$.
Using the logarithmic form of the master functions mitigates their possible contribution to this inefficiency.
For the Jacobian factors, we propose a simple scaling heuristic that is a na\"ive version of
SCLGEN~\cite{MM_scl87}: 
divide each Jacobian by the average of the largest and smallest absolute values of its coefficients.
Scaling serves a secondary role in that it is sometimes necessary to precondition the systems that precompute the
points~$T_j$.   
Without this scaling, we were not always able to compute some of the starting points that lie in the boundary of
$\Delta$ in some of the examples of our test suite.  

Timing of the Khovanskii-Rolle continuation algorithm, information on the number of points found in $T_j$, 
and the number of real solutions for each of the systems in the test suite is archived with our software,
{\tt galeDuality}~\cite{gd}.

\subsection{galeDuality}
\label{SS:software}

Our software package, {\tt galeDuality}, is an open-source program written in {\tt C++} and accepts as input either
a fewnomial system or a Gale system. 
When given a fewnomial system, the software symbolically computes the Gale transform and saves the Gale system to a
file, using the algorithms of Section~\ref{S:algos}, for arbitrary $\ell$.  
Similarly, when given a Gale system, {\tt galeDuality} will transform the system of master functions
into its dual fewnomial system and save the fewnomial system to a file.

If the polyhedron $\Delta$ is unbounded, the software applies the projective transformation of Proposition~\ref{Prop:proj}.
When $\ell = 2$, {\tt galeDuality} calls the maple script of~\cite{BS10} to solve the Gale system using
Khovanskii-Rolle continuation.

After the Khovanskii-Rolle continuation computation is finished, 
{\tt galeDuality} will convert the approximations to solutions of the Gale system back to 
approximations to solutions of the fewnomial system using Algorithm~\ref{alg:unwrapping}, and then call 
{\tt alphaCertified}, if desired by the user, to guarantee that the approximations converge to solutions.  
If the user wishes, {\tt galeDuality} will also certify the approximations in $\Delta$ found by the
Khovanskii-Rolle continuation algorithm.
   
Finally, software and a user manual with more detailed information is available online from each of the authors'
websites.   

%
\bibliographystyle{plain}

\begin{thebibliography}{10}

\bibitem{BBS}
Daniel~J. Bates, F.~Bihan, and F.~Sottile.
\newblock Bounds on {R}eal {S}olutions to {P}olynomial {E}quations.
\newblock {\em Int. Math. Res. Notes}, pages 2007:rnm114--7, 2007.

\bibitem{BHNS}
Daniel~J. Bates, Jonathan~D. Hauenstein, Matthew Niemerg, and Frank Sottile.
\newblock Singular {T}racking in {K}hovanskii-{R}olle {C}ontinuation.
\newblock in preparation.

\bibitem{gd}
Daniel~J Bates, Jonathan~D Hauenstein, Matthew~E Niemerg, and Frank Sottile.
\newblock {\tt galeDuality}: Software for solving fewnomial systems.
\newblock Available at {\tt www.matthewniemerg.com/software/}, 2015.

\bibitem{BertiniBook}
Daniel~J. Bates, Jonathan~D. Hauenstein, Andrew~J. Sommese, and Charles~W.
  Wampler.
\newblock {\em Numerically {S}olving {P}olynomial {S}ystems with {B}ertini}.
\newblock SIAM, 2103.

\bibitem{BS10}
Daniel~J. Bates and F.~Sottile.
\newblock Khovanskii-{R}olle {C}ontinuation for {R}eal {S}olutions.
\newblock {\em Found. Comput. Math.}, 11(5):563--587, 2011.

\bibitem{BD}
F.~Bihan and A.~Dickenstein.
\newblock Descartes' rule of signs for polynomial systems supported on
  circuits, 2014.
\newblock In preparation.

\bibitem{BS}
F.~Bihan and F.~Sottile.
\newblock New {F}ewnomial {U}pper {B}ounds from {G}ale dual {P}olynomial
  {S}ystems.
\newblock {\em Moscow Mathematical Journal}, 7(3):387--407, 2007.

\bibitem{BS_Gale}
F.~Bihan and F.~Sottile.
\newblock Gale duality for {C}omplete {I}ntersections.
\newblock {\em Ann. Inst. Fourier (Grenoble)}, 58(3):877--891, 2008.

\bibitem{Craciun}
Gheorghe Craciun, Luis Garc{\'{\i}}a-Puente, and Frank Sottile.
\newblock Some geometrical aspects of control points for toric patches.
\newblock In {\em Mathematical Methods for Curves and Surfaces}, volume 5862 of
  {\em Lecture Notes in Computer Science}, pages 111--135. Springer, Berlin,
  Heidelberg, New York, 2010.

\bibitem{Ge01}
Kurt Georg.
\newblock Improving the {E}fficiency of {E}xclusion {A}lgorithms.
\newblock {\em Adv. Geom.}, 1(2):193--210, 2001.

\bibitem{Ge03}
Kurt Georg.
\newblock A {N}ew {E}xclusion {T}est.
\newblock {\em J. Comput. Appl. Math.}, 152(1-2):147--160, 2003.

\bibitem{H}
Jonathan~D. Hauenstein.
\newblock Numerically {C}omputing {R}eal {P}oints on {A}lgebraic {S}ets.
\newblock {\em Acta applicandae mathematicae}, 125(1):105--119, 2013.

\bibitem{HS}
Jonathan~D. Hauenstein and F.~Sottile.
\newblock Algorithm 921: alpha{C}ertified: {C}ertifying {S}olutions to
  {P}olynomial {S}ystems.
\newblock {\em ACM Trans. Math. Softw.}, 38(4):28, 2012.

\bibitem{joseph60}
R.~D. Joseph.
\newblock The number of orthants in $n$-space intersected by a $s$-dimensional
  subspace.
\newblock {\em Technical Memorandum 8, Project PARA}, 1960.

\bibitem{Kh91}
A.G. Khovanskii.
\newblock {\em Fewnomials}.
\newblock Trans. of Math. Monographs, 88. AMS, 1991.

\bibitem{LLR}
Jean~Bernard Lasserre, Monique Laurent, and Philipp Rostalski.
\newblock Semidefinite {C}haracterization and {C}omputation of
  {Z}ero-dimensional {R}eal {R}adical {I}deals.
\newblock {\em Found. Comput. Math.}, 8(5):607--647, 2008.

\bibitem{LLL}
A.~K. Lenstra, H.~W. Lenstra, Jr., and L.~Lov{\'a}sz.
\newblock Factoring {P}olynomials with {R}ational {C}oefficients.
\newblock {\em Math. Ann.}, 261(4):515--534, 1982.

\bibitem{MM_scl87}
Keith Meintjes and Alexander~P. Morgan.
\newblock A {M}ethodology for {S}olving {C}hemical {E}quilibrium {S}ystems.
\newblock {\em Applied Mathematics and Computation}, 22(4):333--361, 1987.

\bibitem{MFRCSD}
Stefan M\"uller, Elisenda Feliu, Georg Regensburger, Carsten Conradi, Anne
  Shiu, and Alicia Dickenstein.
\newblock Sign conditions for injectivity of generalized polynomial maps with
  applications to chemical reaction networks and real algebraic geometry, 2013.
\newblock {\tt arXiv:1311.5493}.

\bibitem{RRS}
F.~Rouillier, M.-F. Roy, and M.~Safey El~Din.
\newblock Finding at least {O}ne {P}oint in each {C}onnected {C}omponent of a
  {R}eal {A}lgebraic {S}et {D}efined by a {S}ingle {E}quation.
\newblock {\em J. Complexity}, 16(4):716--750, 2000.

\bibitem{RUR}
Fabrice Rouillier.
\newblock Solving {Z}ero-dimensional {S}ystems through the {R}ational
  {U}nivariate {R}epresentation.
\newblock {\em Appl. Algebra Engrg. Comm. Comput.}, 9(5):433--461, 1999.

\bibitem{SW05}
Andrew~J. Sommese and Charles~W. Wampler, II.
\newblock {\em The {N}umerical {S}olution of {S}ystems of {P}olynomials
  {A}rising in {E}ngineering and {S}cience}.
\newblock World Scientific Publishing Co. Pte. Ltd., Hackensack, NJ, 2005.

\end{thebibliography}

\end{document}